\documentclass[12pt,twoside]{amsart}   % Standard LaTeX, amsart
\usepackage{amsthm, amsmath,amssymb, amsfonts, fancyhdr,mathrsfs}
\usepackage{amscd}
\usepackage{graphicx}
\usepackage{psfrag}
\usepackage{setspace}

\input xy
\xyoption{all}

\theoremstyle{plain}
\newtheorem{thm}{Theorem}[section]
\newtheorem{lem}[thm]{Lemma}
\newtheorem{prop}[thm]{Proposition}
\newtheorem{cor}[thm]{Corollary}

\theoremstyle{definition}
\newtheorem{defn}{Definition}[section]
\newtheorem{exmp}{Example}[section]

\newtheorem{defprop}{Definition-Proposition}[section]

\newcommand{\parb}[1]{\left(#1\right)}

\begin{document}

\title[Rational Functions over $p$-adic Numbers]{Digraph Representations of Rational Functions over the  $p$-adic Numbers}

\author[Diao] {Hansheng Diao}
\address[Hansheng Diao]{Massachusetts Institute of Technology, MA 02139, USA }
\email{hansheng@mit.edu}
\author[Silva]{Cesar E. Silva}
\address[Cesar Silva]{Department of Mathematics\\
     Williams College \\ Williamstown, MA 01267, USA}
\email{csilva@williams.edu}

\subjclass{Primary 37A05; Secondary
37F10} \keywords{p-adic dynamics, ergodic, minimal, locally scaling}

\date{\today}

\begin{abstract}
In this paper, we construct a digraph structure on $p$-adic
dynamical systems defined by rational functions. We study the conditions under which
 the functions are measure-preserving,
 invertible and isometric, ergodic,   and minimal on invariant subsets,
 by means of graph theoretic properties.
\end{abstract}

\maketitle

%%%%%
\section{Introduction}

In recent years there has been interest in studying the measurable and topological dynamical properties of maps defined on compact open subsets of the $p$-adic numbers $\mathbb Q_p$.  In \cite{A06}, Anashin characterizes  measure-preserving and ergodic properties of $1$-Lipshitz maps on $\mathbb Z_p$, extending the work of other authors in \cite{OsZi75}, \cite{CoelhoParry}, \cite{GKL}, \cite{BrykSilva}. Minimality of a class of maps of the form $T_{\alpha,\beta}(z)=\alpha z+\beta, \alpha,\beta\in\mathbb Z_p$ was later studied
in \cite{FLY07}. Non-1-Lipshitz maps have also been studied; the reader may refer to \cite{KLPS} and the recent monograph \cite{AnKh}, and the references therein. In \cite{KLPS}, the authors introduced the notion of locally scaling transformations on compact-open subsets of a non-archimedean local field, showed when they are measure-preserving for Haar measure, proved a structure theorem about them  and studied ergodic properties of those maps such as mixing.  In this paper we are interested in rational functions defined on compact, or locally compact, open, subsets of $\mathbb Q_p$. We generalize the notion of locally scaling transformations and also consider locally $1$-Lipschitz maps. To study the dynamics of these maps we associate with the maps a digraph structure. Using the digraph we characterize the measure-preserving property in Theorem~\ref{m-p} and ergodicity and minimality in Theorem~\ref{erg}.
In Section~\ref{S:subsidiary} we introduce the notion of a subsidiary digraph, which helps us to determine the measure-preserving component of a rational function. In Section ~\ref{S:qp} we give a characterization of invertible locally isometric rational functions over $\mathbb{Q}_p$. We also prove that for locally 1-Lipschitz rational functions over $\mathbb{Q}_p$,  the measure-preserving property is equivalent to the map being an invertible local isometry.

\subsection{Acknowledgements}
This paper is based on research as part of the  Ergodic Theory group of the
2008 SMALL summer research project at Williams College.  Support for
the project was provided in part by National Science Foundation REU Grant DMS
- 0353634 and by  the Bronfman Science Center of Williams College.

%%%%%
\section{Preliminaries}
\bigskip
We  will study  rational functions of the form
$f(x)=\frac{P(x)}{Q(x)}$ where $P(x)$ and $ Q(x)$ are in $\mathbb{Q}_p[x]$. Without loss of generality, we
can assume that $P(x), Q(x)\in \mathbb{Z}_p[x]$ and $P(x), Q(x)$ are
coprime. Equip $\mathbb{Q}_p$ with the usual metric and topology (see \cite{R00}). Define
\[
B_r(a) = \{x\in \mathbb{Q}_p\,|\, |x-a|\leq r\}\]
 and \[S_r(a) = \{x\in \mathbb{Q}_p\,|\, |x-a|=r\}.\] Let $\mu$ be the usual Haar measure; i.e., let $\mu(B_{p^l}(a))=p^l$ for
all $l\in \mathbb{Z}$ and $a\in \mathbb{Q}_p$.

We will consider a well-defined rational function  $f:X\rightarrow X$  for some $X\subset \mathbb{Q}_p$, usually open and locally compact, or open and  compact. Then we have an
 dynamical system $(X, \mathscr{B}(X), \mu, f)$, possibly infinite. Here we require that $Q(x)$ has no root in $X$.

First we consider the case when $X=\mathbb{Q}_p$. Given a dynamical system $(\mathbb{Q}_p, \mathscr{B}(\mathbb{Q}_p),
\mu, f)$ we are interested in whether it is
minimal, ergodic, weakly mixing, invertible, isometric, etc.
But we note that  there is no minimal or ergodic rational function over $\mathbb{Q}_p$. As they are not the our main object of study, the proofs
of the following Propositions ~\ref{no-minimal} and ~\ref{no-ergodic} are presented in the appendix.

\begin{prop}\label{no-minimal}
There is no minimal rational function over $\mathbb{Q}_p$.
\end{prop}

\begin{prop}\label{no-ergodic}
There is no measure-preserving and ergodic rational function over $\mathbb{Q}_p$.
\end{prop}

\noindent\textbf{Remark.} There exist invertible isometric rational functions over $\mathbb{Q}_p$. For example, $f(x)=ax+b$ for $a\in \mathbb{Z}_p^{\times}$ and $b\in \mathbb{Q}_p$. Another example is given in Example ~\ref{exmp-inv-qp}. Actually, to check whether a rational function is an invertible local isometry, we only need to check it on a compact open subset. We will give a characterization of such rational functions in Section ~\ref{S:qp}.\\

Now the next natural question is: On which subsets of $\mathbb Q_p$ is the rational map
minimal, ergodic, invertible, or isometric? In particular, we want to find
a subset $Y\subset \mathbb{Q}_p$ such that $f|_Y:Y\rightarrow Y$ is
well-defined and satisfies some of these dynamical properties.

The following propositions shows that rational functions locally
look like scaling functions.  The definition of \emph{locally scaling} below
extends  the one in ~\cite{KLPS} as here  the local scalar is
not necessarily greater than 1 .

\begin{defn}\label{local-scal}
Let $X$ be any open subset of $\mathbb{Q}_p$ and let $g:X\rightarrow
X$ be a well-defined measurable function. We say that $g$ is
\textbf{locally scaling} if for any $a\in X$ there exists $r=r(a)>0$
and $C=C(a)>0$ such that 
\[
|g(x)-g(y)|=C|x-y|\text{ whenever }x, y\in
B_r(a).\]
 The map $C:X\rightarrow \mathbb{R}_{>0}$ is called the
\textbf{scaling function}. We also call $C(a)$ the \textbf{local
scalar} at $a$.
\end{defn}

\begin{prop}\label{P:local-scal}
Let $X\subset\mathbb{Q}_p$ be an open subset and let $f:X\rightarrow
X$ be a rational function. Let $a\in X$.
\begin{enumerate}
\item If $f'(a)\neq 0$, then there exists $r>0$ such that $B_r(a)\subset
X$ and $|f(x)-f(y)|=|f'(a)|\cdot|x-y|$ whenever $x,y\in B_r(a)$.
\item If $f'(a)=0$, then for any $r_0>0$, there exists $r>0$ such
that $B_r(a)\subset X$ and $|f(x)-f(y)|< r_0|x-y|$.
\end{enumerate}
\end{prop}

\begin{proof} (Also see~\cite{KN04} Ch.3 Lemma 1.6.)
Recall that rational functions over $\mathbb{Q}_p$ are analytic
outside their poles (see~\cite{R00} Ch.6). Take $R>0$ such that
$B_R(a)\subset X$. Then we can write
\[f(x)-f(a) = \sum_{i=1}^{\infty}\frac{(x-a)^i}{i!}\frac{d^if}{dx^i}(a)\]
for $x\in B_R(a)$.\\
Thus
\begin{equation*}
\begin{split}
f(x)-f(y) & = (f(x)-f(a))-(f(y)-f(a))\\
& = (x-y)f'(a) +
\sum_{i=2}^{\infty}\frac{(x-a)^i-(y-a)^i}{i!}\frac{d^if}{dx^i}(a)\\
& = (x-y)\big[f'(a)+T(x,y,a)\big]\\
\end{split}
\end{equation*}
where \[T(x,y,a)
=\sum_{i=2}^{\infty}\frac{1}{i!}\frac{d^if}{dx^i}(a)[(x-a)^{i-1}+(x-a)^{i-2}(y-a)+\cdots+(y-a)^{i-1}].\]
By the analyticity of $f$ on $B_R(a)$, we know that the $R$-Gauss
norm $||f||_R = \sup_{0\leq
i<\infty}|\frac{1}{i!}\frac{d^if}{dx^i}(a)|R^i$ is finite.
\begin{enumerate}
\item If $f'(a)\neq 0$, take $0<r<\min\{R, \frac{|f'(a)|}{||f||_R}\}$. Then
for any $x,y\in B^-_r(a)$, \[|T(x,y,a)| \leq  \sup_{2\leq
i<\infty}|\frac{1}{i!}\frac{d^if}{dx^i}(a)|r^{i-1} \leq
r^{i-1}\frac{||f||_R}{R^i} < r||f||_R < |f'(a)|\]
\item Similarly, if $f'(a)=0$, take $0< r< \min\{R, \frac{r_0}{||f||_R}\}$.
\end{enumerate}
\end{proof}

\noindent{\bf Remark.} Under the assumption of Proposition ~\ref{P:local-scal}(1), it is not hard to see that $|f'(x)|=|f'(a)|$ for all $x\in B_r(a)$. Moreover, for any $r_1\leq r$, $f$ maps the ball $B_{r_1}(a)$ into a ball of radius $r_1/|f'(a)|$. So, if $f$ is measure-preserving, we must have $|f'(a)|\geq 1$ for all $a\in X$.

\begin{cor}\label{cor-local-scal}
Let $X$ be an open subset of $\mathbb{Q}_p$ and let $f:X\rightarrow
X$ be a rational function. In addition, we assume that $f'(x)$ has
no root in $X$. Then $f$ is locally scaling with local scalar
$C(a)=f'(a)$ for all $a\in X$.
\end{cor}

\noindent{\bf Remark.} The condition ``\emph{$f'(x)$ has no root
in $\mathbb{Q}_p$}'' is crucial here. From
Proposition~\ref{P:local-scal}(2) we see that $f$ is never
locally scaling around a root of $f'(x)$.\\

We will mainly be interested in the case where
$X\subset\mathbb{Q}_p$ is  a compact open subset of $\mathbb Q_p$. As before,  we assume
$f:X\rightarrow X$ is well-defined and consider the dynamical system
$(X, \mathscr{B}(X), \mu, f)$.

\begin{defprop}\label{unif-local-scal}
Let $X$ be a compact open subset of $\mathbb{Q}_p$ and let
$f:X\rightarrow X$ be a rational function. Assume, in addition, 
that $f'(x)$ has no roots in $X$. Then $f$ is \textbf{uniformly
locally scaling} with local scalar $C(a)=|f'(a)|$; i.e., there exist $r>0$ such that for any $a\in
X$, $|f(x)-f(y)|=|f'(a)|\cdot|x-y|$ whenever $x,y\in B_r(a)$.
\end{defprop}

\begin{proof}
$X$ is covered by the union of balls $\bigcup_{a\in
X}B_{r(a)}(a)$. Since $X$ is compact, we can find a finite
subcover $X = \bigcup_{i=1}^k B_{r(a_k)}(a_k)$. Take
$r=\min_{1\leq i\leq k}r(a_k)$.
\end{proof}

Similarly, we can define the notions \emph{(uniformly) locally
isometric, (uniformly) locally $\rho$-Lipschitz, (uniformly) locally
bounded scaling} as follows.

\begin{defn}\label{local-isom}
A map $f:X\rightarrow X$ is \emph{\textbf{uniformly locally
isometric}} if there exists a constant $r>0$ such that for all $a\in
X$, $|f(x)-f(y)| = |x-y|$ whenever $x,y\in B_r(a)$.
\end{defn}

\begin{defn}\label{local-1-Lip}
Let $\rho>0$. A map $f:X\rightarrow X$ is \emph{\textbf{uniformly
locally $\rho$-Lipschitz}} if there exists a constant $r>0$ such
that for all $a\in X$, $|f(x)-f(y)|\leq \rho|x-y|$ whenever $x,y\in
B_r(a)$.

\end{defn}

\begin{defn}
A map $f:X\rightarrow X$ is \emph{\textbf{uniformly locally bounded
scaling}} if there exists a constant $C>0$ such that $f$ is
uniformly locally scaling and $C(a)\leq C$ for all $a\in X$.
\end{defn}

By a similar argument as in Proposition~\ref{unif-local-scal}, we
have the following criteria.
\begin{prop}\label{criteria}
Let $X$ be a compact open subset of $\mathbb{Q}_p$ and let
$f:X\rightarrow X$ be a rational function. Then
\begin{itemize}
\item $f$ is uniformly locally isometric if and only if $|f'(a)|= 1$ for all $a\in X$.
\item $f$ is uniformly locally $\rho$-Lipschitz if and only if $|f'(a)|\leq \rho$ for all $a\in X$.
\item $f$ is uniformly locally bounded scaling if and only if $|f'(a)|$ is bounded on
$X$ and $f'(x)$ has no root in $X$.

\end{itemize}
\end{prop}

\begin{proof}
It follows immediately from Proposition~\ref{P:local-scal} and the
compactness of $X$.
\end{proof}

%%%%%%
\section{Digraph Structure for Locally 1-Lipschitz Functions}\label{S:digraph}

Throughout this section we  let $f(x)$ be a locally 1-Lipschitz rational function
on a compact open subset $X\subset \mathbb{Q}_p$. By Proposition ~\ref{criteria}, this is equivalent to saying $|f'(a)|\leq 1$ for all $a\in X$.

Let $r=p^l$ ($l\in \mathbb{Z}$) be the constant involved in
Definition~\ref{local-1-Lip}. Let $t$ be any integer less than or
equal to $l$. Recall that we can write $X$ uniquely as disjoint
union of finitely many closed balls of radius $p^t$. Write, $X =
\bigsqcup_{i=1}^m D_{t,i}$. Since $f$ is 1-Lipschitz in each ball of
radius $p^l$, $f$ maps each  ball $D_{t,i}$ into another ball
$D_{t,j}$.

Construct a digraph $G(f, p^t)$ as follows. Let the set of vertices be  \[V(G) = \{A_{t,1},
A_{t,2},\cdots, A_{t,m}\}\] where $m=m(t)=\mu(X)/p^t$.  Join
$A_{t,i}$ to $A_{t,j}$ with a directed  edge  if $f$ maps $D_{t,i}$ to $D_{t,j}$; i.e.,
\[E(G)=\{(A_{t,i},A_{t,j}) \,|\, f(D_{t,i}) \subset D_{t,j}\}.\]

It is clear that the outdegree $d^+_G(A_{t,i})=1$ for all
$1\leq i\leq m$. Hence $\# E(G(f, p^t))=m(t)$ and there exists at
least one cycle in the digraph $G(f, p^t)$.

\begin{exmp}\label{examp}
Let $X = B_{1/7}(2)\cup B_{1/7}(5)\subset \mathbb{Q}_7$ and consider $f(x) =
\frac{x^2-1}{x}$. Then $f'(x) = \frac{x^2+1}{x^2}$. Since
$\big(\frac{-1}{7}\big)=-1$, we have $|x^2+1|=1$ for all $x\in X$.
Thus, $f$ is uniformly locally isometric and we can take $r=7^{-1}$.
\end{exmp}

The dynamics of $f$ on $7^{-2}$-balls is shown in Figure~\ref{F:sevenballs}. The corresponding digraph is shown in Figure~\ref{F:digraph}. The digraph is the disjoint union of three cycles of length 2,6,6, respectively.

\begin{figure}[htb]

\begin{center}
    \psfrag{A}[]{$B_{7^{-2}}(2)$}
    \psfrag{B}[]{$B_{7^{-2}}(9)$}
    \psfrag{C}[]{$B_{7^{-2}}(16)$}
    \psfrag{D}[]{$B_{7^{-2}}(23)$}
    \psfrag{E}[]{$B_{7^{-2}}(30)$}
    \psfrag{F}[]{$B_{7^{-2}}(37)$}
    \psfrag{G}[]{$B_{7^{-2}}(44)$}
    \psfrag{a}[]{$B_{7^{-2}}(5)$}
    \psfrag{b}[]{$B_{7^{-2}}(12)$}
    \psfrag{c}[]{$B_{7^{-2}}(19)$}
    \psfrag{d}[]{$B_{7^{-2}}(26)$}
    \psfrag{e}[]{$B_{7^{-2}}(33)$}
    \psfrag{f}[]{$B_{7^{-2}}(40)$}
    \psfrag{g}[]{$B_{7^{-2}}(47)$}
    \psfrag{P}[]{$B_{7^{-1}}(2)$}
    \psfrag{Q}[]{$B_{7^{-1}}(5)$}
 \includegraphics[width=13cm]{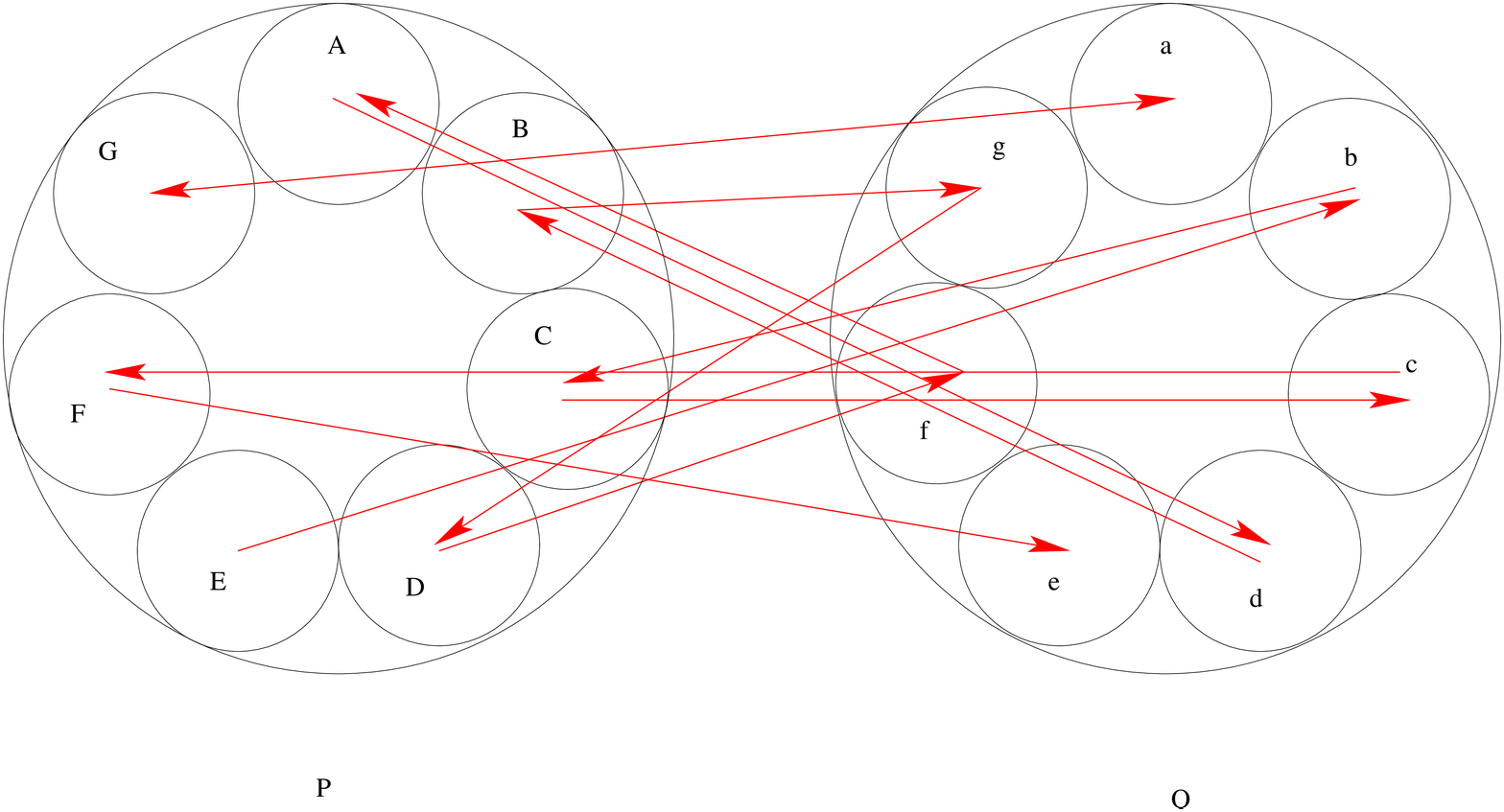}\\

\end{center}
\caption { }
 \label{F:sevenballs}
\end{figure}

\begin{figure}[htb]

\begin{center}
\psfrag{K}[]{$G(f,7^{-2})$}
 \includegraphics[width=6.5cm]{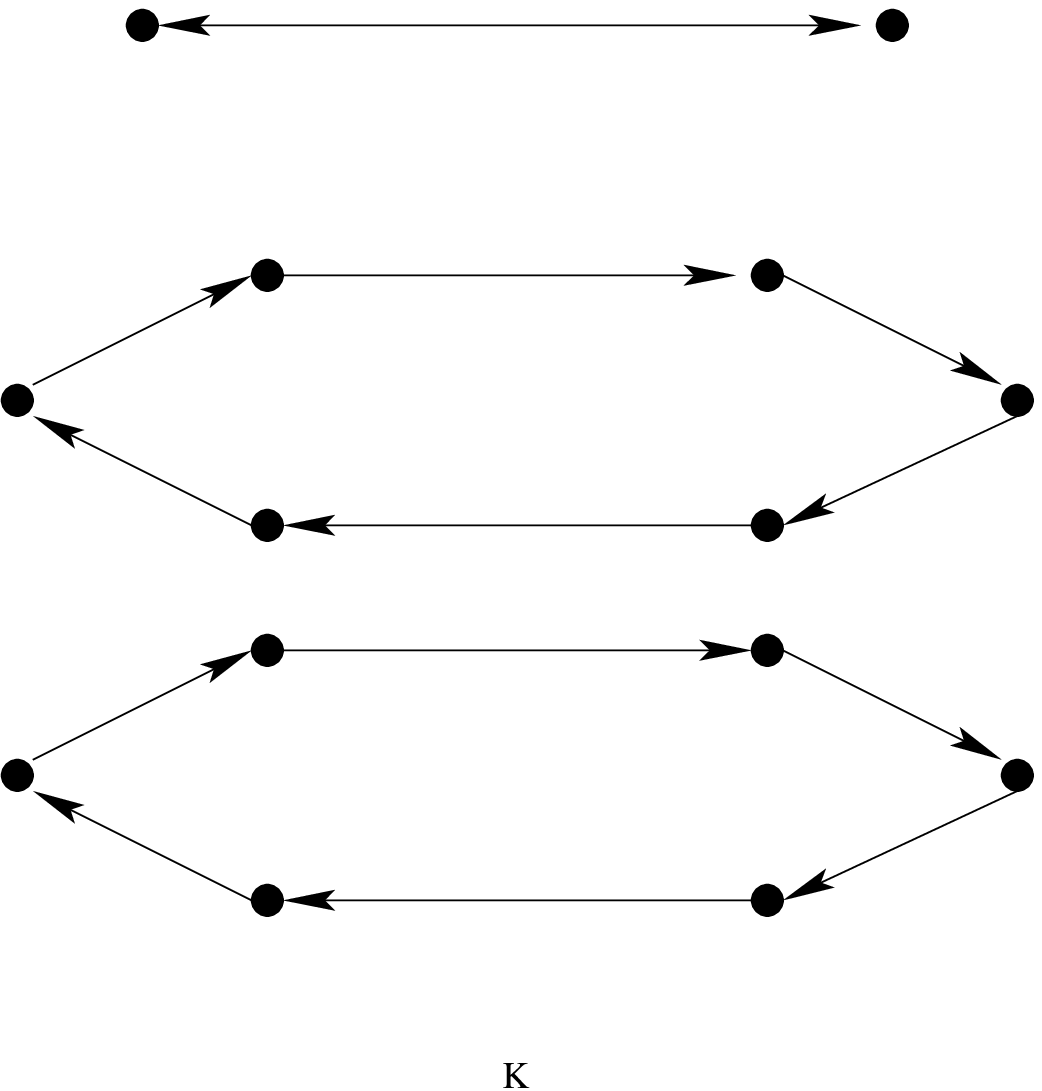}\\

\end{center}
\caption { }
 \label{F:digraph}
\end{figure}

The following result gives a characterization of measure-preserving
locally 1-Lipschitz maps. It serves as a generalization of Corollary
2.4 in~\cite{A06}.

\begin{thm}\label{m-p}
Let $X$ be a compact open subset of $\mathbb{Q}_p$ and let
$f:X\rightarrow X$ be a locally 1-Lipschitz rational function. The following
statements are equivalent.
\begin{enumerate}
\item $f$ is measure-preserving.
\item $f$ is  invertible.
\item $f$ is invertible and locally isometric.
\item For all $t\leq l$, $G(f, p^t)$ is a disjoint union of cycles.
\end{enumerate}
\end{thm}

\begin{proof}
(1)$\Rightarrow$(4): Note that $\sum_{i=1}^m
d^-(A_{t,i})=\sum_{i=1}^m d^+(A_{t,i})=m$.
Now suppose that $d^-(A_{t,i})\geq 2$ for some $i$. Say, \[(A_{t,j_1},
A_{t,i}), (A_{t,j_2}, A_{t,i}) \in E(G(f, p^t)).\] Then
$D_{t,j_1}\cup D_{t, j_2}\subset f^{-1}(D_{t,i})$. Thus
\[\mu(D_{t,i})=\mu(f^{-1}(D_{t,i}))\geq
\mu(D_{t,j_1})+\mu(D_{t,j_2})=2\mu(D_{t,i}),\] a contradiction.
So $d^-(A_{t,i})=1$ for all $1\leq i\leq m$. By simple graph theory, if every vertex in a digraph has indegree
and outdegree 1, then the digraph must be a disjoint union of
cycles.\\

\noindent(4)$\Rightarrow$(2): Let $a,b\in X$. There exists an
integer $t<l$ such that $|a-b|>p^t$. Then $a,b$ correspond to
different vertices in the digraph $G(f, p^t)$. Since $G(f, p^t)$ is
a disjoint union of cycles, the images of $a, b$ also correspond to
different vertices. In particular, $f(a)\neq f(b)$.\\
It remains to show that $f$ is surjective. Let $a\in X$. Since every
vertex in the digraph $G(f,p^t)$ has indegree 1, there exists a unique
disk $D_{t,i_t}$ such that $f(D_{t,i_t})\subset B_{p^t}(a)$. We
obtain an infinite nested sequence $D_{l, i_l}\supset D_{l-1,
i_{l-1}}\supset \cdots$. Note that $\lim_{t\rightarrow
-\infty}\mu(D_{t,i_t})=0$. By the completeness of $\mathbb{Q}_p$,
$\bigcap_{t\leq l}D_{t,i_t}$ consists of a single point. This point
is the inverse image of $a$.\\

\noindent(2)$\Rightarrow$(4): Since $f$ is bijective, each vertex
has indegree $d^-\geq 1$. Hence $d^-(A_{t,i})=1=d^+(A_{t,i})$ for
all $i$. Thus $G(f,p^t)$ is a disjoint union of cycles.\\

\noindent[(2) and (4)$]\Rightarrow$(3): We show that $f$ is locally isometric with $r = p^l$.\\
Let $x,y\in X$ with $|x-y|=p^t$ ($t\leq l$). Then
$f(B_{p^t}(y))\subset B_{p^t}(f(y))$. Since the digraph $G(f,
p^t)$ is a disjoint union of cycles, each vertex has indegree 1.
Thus $f^{-1}(B_{p^t}(f(y)))\subset B_{p^t}(y)$. Hence $f:
B_{p^t}(y)\rightarrow B_{p^t}(f(y))$ is a bijection. Similarly,
$f:B_{p^{t-1}}(y)\rightarrow B_{p^{t-1}}(f(y))$ is a bijection.
So $f(S_{p^t}(y))=S_{p^t}(f(y))$. In particular, $f(x)\in
S_{p^t}(f(y))$; i.e., $|f(x)-f(y)|=p^t=|x-y|$.\\

\noindent(3)$\Rightarrow$(2): Trivial.\\

\noindent[(3) and (4)$]\Rightarrow$(1): Since $X$ is compact, $f$ is
uniformly locally isometric. Let $r=p^l$ be the constant involved in
Definition ~\ref{local-isom}. Then for any integer $t\leq l$ and
any $a\in X$, we have $B_{p^t}(f^{-1}(a))= f^{-1}(B_{p^t}(a))$
(similar to the ``(2)$\Rightarrow$(3)'' part). Hence
$\mu(f^{-1}(B_{p^t}(a)))=\mu(B_{p^t}(a))$. This means $f$
preserves the measure of balls. Therefore, $f$ is
measure-preserving.
\end{proof}

We also have the following characterization of ergodic locally
1-Lipschitz maps. This is a generalization of Proposition 4.1
in~\cite{A06}.

\begin{thm}\label{erg}
Let $X$ be a compact open subset of $\mathbb{Q}_p$ and let
$f:X\rightarrow X$ be a locally 1-Lipschitz rational function. Then the following
statements are equivalent.
\begin{enumerate}
\item $f$ is measure-preserving and ergodic.
\item $f$ is minimal.
\item For all $t\leq l$, $G(f, p^t)$ consists of a single cycle.
\end{enumerate}
\end{thm}

\begin{proof}
(1)$\Rightarrow$(2):  By Theorem~\ref{m-p}, we know that $f$ is invertible and
(uniformly) locally isometric. Let $r=p^l$ be the corresponding
constant defined in Definition ~\ref{local-isom}.

Let $\varepsilon>0$ and let $x,y\in X$. Consider $B_1=B_{\rho}(x)$
and $B_2=B_{\rho}(y)$ with $0< \rho < \min\{\varepsilon, p^l\}$. By the
ergodicity of $f$, there exists $n>0$ such that $\mu(f^n(B_1)\cap
B_2)>0$. As $f$ is locally isometric in each $\rho$-ball, we have
$f^n(B_1)\subset B_{\rho}(f^n(x))$. Thus $|f^n(x)-y|\leq
\rho<\varepsilon$. By the arbitrariness of $\varepsilon$, $x$, and
$y$, we conclude that $f$ is minimal.\\

\noindent(2)$\Rightarrow$(3): Let $A_{t,1}, \cdots, A_{t,m}$ denote
the vertices of $G(f,p^t)$. Suppose $d^-(A_t,i)=0$ for some $i$.
Then $D_{t,i}\cap \{f^n(x)\,|\,n\geq 0 \}=\emptyset$, for all $x\in
X\backslash D_{t,i}$, which contradicts the minimality of $f$.
So $d^-(A_{t,i})\geq 1$ for all $1\leq i\leq m$. Since $\sum_{i=1}^m
d^-(A_{t,i})=m$, $d^-(A_{t,i})=1$ for all $i$. Thus $G(f,p^t)$ is a
disjoint union of cycles. By the minimality of $f$, there is only one
cycle.\\

\noindent(3)$\Rightarrow$(1): By Theorem~\ref{m-p}, $f$ is
invertible, measure-preserving and (uniformly) locally isometric.
Let $r=p^l$ be the corresponding constant defined in Definition ~\ref{local-isom}.

Now we show that $f$ is uniquely ergodic, and thus ergodic. We need to show that the normalized Haar measure $\bar{\mu}$ (i.e., $\bar{\mu}:=\mu/\mu(X)$) is the only $f$-invariant probability measure on $X$. Let $\nu$ be any $f$-invariant probability measure on $X$. Consider a ball $B$ of radius $r=p^t\leq p^l$. Since $G(f, p^t)$ consists of a single cycle and $f$ is bijective, we can write $X=\bigsqcup_{i=0}^{n-1}f^i(B)$ where $n = \mu(X)/\mu(B)=1/\bar{\mu}(B)$. So $1=\sum_{i=0}^{n-1}\nu(f^i(B))=n\nu(B)$. Then $\nu(B)=1/n=\bar{\mu}(B)$. This means $\bar{\mu}$ and $\nu$ agree on all balls of radii $\leq p^l$. Therefore, $\nu=\bar{\mu}$, as desired.
\end{proof}

We now see some applications. By Theorem~\ref{m-p} (respectively, 
Theorem~\ref{erg}), to show that $f$ is not measure-preserving
(respectively, not ergodic), it suffices to find $t\leq l$ such that
$G(f,p^t)$ is not a union of cycles (respectively, not a single cycle). All
these can be done by computer within a reasonable time.

\begin{exmp}
Let $X$ and $f$ be the same as in example ~\ref{examp}. We have already seen that $G(f, 7^{-2})$ consists of three disjoint cycles. So $f$ is not ergodic on $X$. We will show in the next section that $f$ is measure-preserving.
\end{exmp}

%%%%%
\section{Subsidiary Digraph of Locally 1-Lipschitz Rational Functions}\label{S:subsidiary}

By Theorem ~\ref{m-p}, a locally 1-Lipschitz rational function is measure-preserving if and only if the digraph $G(f,p^t)$ is a disjoint union of cycles for every $t\leq l$. Practically, it is impossible to check infinitely many digraphs. In this section, we construct a subsidiary digraph $G^{\ast}$ for locally
1-Lipschitz rational function $f$ on compact open sets $X\subset
\mathbb{Q}_p$. With the help of $G^{\ast}$, we only need to work on finitely many of $G(f, p^t)$.

Throughout this section, we assume that $f$ is locally 1-Lipschitz on $X$. By Proposition~\ref{criteria}, this is equivalent to
\[|f'(a)|=\frac{|P'(a)Q(a)-P(a)Q'(a)|}{|Q(a)|^2} \leq 1\text{ for all }a\in X.\]

We keep the notations of $r, l, t, D_{t,i}, A_{t,i}$ as in Section ~\ref{S:digraph}. By definition, $|f(x)-f(y)|=|f'(a)|\cdot|x-y|$ whenever $|x-y|\leq p^l$ and $f'$
has no root in $B_{p^l}(x)$. Define $V(G^{\ast}) = \{A_{t,1},
A_{t,2}, \cdots, A_{t,m}\}$. We construct directed edges as follows:

\begin{enumerate}
\item For each $t\leq l$, choose a set of points $S_t=\{a_{t,1}, a_{t,2},\cdots, a_{t,m}\}\subset
X$ such that
\begin{itemize}
\item $a_{t,i}\in D_{t,i}$ ($1\leq i\leq m$), and
\item $S_l\subset S_{l-1}\subset S_{l-2}\subset \cdots$.
\end{itemize}
The elements of $S_t, t\leq l$ are called the \textbf{representatives}.

\item Suppose $f(D_{t,i})\subset D_{t,j}$. Choose the least $s\in \mathbb{Z}_{\geq
0}$ such that \[P(p^sx+a_{t,i})-(p^sy+a_{t,j})Q(p^sx+a_{t,i})\in
\mathbb{Z}_p[x,y].
\]
Then we join $A_{t,i}$ to $A_{t,j}$ if and only if
\[p^t<\min\left\{p^{-s}, p^l|f'(a_{t,i})|, \frac{|Q(a_{t,i})|\cdot|f'(a_{t,i})|}{|Q'(a_{t,i})|}, p^{-2s}|Q(a_{t,i})|\cdot|f'(a_{t,i})|^2\right\}.\]
\end{enumerate}

\noindent\textbf{Remark.} The construction of $G^{\ast}(f,p^t)$
depends on the choice of $S_t$'s.\\

The digraph $G^{\ast}(f, p^t)$ is called the \textbf{subsidiary
digraph} of $f$. It is clear that $G^{\ast}(f, p^t)$ is a subgraph
of $G(f,p^t)$. The following proposition shows that, under certain
condition, $G^{\ast}$ coincides with $G$ for sufficiently small
$t$.

\begin{prop}\label{intrinsic-level}
Suppose $f'(x)$ has no root in $X$. Then there exists $t\leq l$
such that $G(f, p^t)= G^{\ast}(f, p^t)$.
\end{prop}

The largest such $t$ is called the \textbf{intrinsic level} of $f$. We denote it by $t_0$.

In order to prove the proposition, we need the following lemma.

\begin{lem}\label{local-same-norm}
Let $F(x)\in \mathbb{Q}_p[x]$ be a polynomial and let $a\in X$.
\begin{enumerate}
\item\label{one} If $F(a)\neq 0$, then there exists $r>0$ such that
$|F(x)|=|F(a)|$ whenever $a\in B_r(a)$.
\item\label{two} If $F(a)=0$, then for any $r_0>0$ there exists $r>0$ such that
$|F(x)|\leq r_0$ whenever $a\in B_r(a)$.
\end{enumerate}
\end{lem}

\begin{proof}
For part \eqref{one}, if $F'(a)\neq0$, by Proposition~\ref{P:local-scal}(1), there
exists $\rho>0$ such that $|F(x)-F(a)|=|F'(a)|\cdot|x-a|$ for all
$x\in B_{\rho}(a)$. Take $0<r<\min\{\rho, |F(a)|/|F'(a)|\}$. Then
$|F(x)-F(a)|<|F(a)|$ whenever $x\in
B_r(a)$. Thus $|F(x)|=|F(a)|$.

Now suppose $F'(a)=0$. By Proposition~\ref{P:local-scal}(2), there
exists $\rho>0$ such that $|F(x)-F(a)|<|x-a|$ for all $x\in
B_{\rho}(a)$. Take \[0<r<\min\{\rho, |F(a)|\}.\] Then
$|F(x)-F(a)|<|F(a)|$ whenever $x\in B_r(a)$. Thus $|F(x)|=|F(a)|$.
 The proof part \eqref{two} is similar.
\end{proof}

\begin{proof}[Proof of Proposition~\ref{intrinsic-level}]
For any $a_1, a_2\in X$, let $s(a,b)$ be the least nonnegative
integer $s$ such that
\[T_{a_1,a_2,s}(x,y)=P(p^sx+a_1)-(p^sy+a_2)Q(p^sx+a_1)\in
\mathbb{Z}_p[x,y].\]

Let $a\in X$. Then there exists $\rho_1(a)>0$ such that $s(a',b)$ is
constant for $a'\in B_{\rho_1(a)}(a)$ and $b\in
B_{\rho_1(a)}(f(a))$. We denote this constant by $s=s(a)$.\\

\textit{Claim}: For any $a\in X$, there exists $r(a) =
p^{t(a)}\leq p^l$ such that $r(a)<\min\left\{\rho_1(a), p^{-s},
p^l|f'(a')|, \frac{|Q(a')|\cdot|f'(a')|}{|Q'(a')|},
p^{-2s}|Q(a')|\cdot|f'(a')|^2 \right\}$ for all $a'\in
B_{r(a)}(a)$.

Let $F(x)=P'(x)Q(x)-P(x)Q'(x)$. Then $F(a)=f'(a)\cdot Q(a)^2\neq
0$.
If $Q'(a)\neq 0$, by Lemma~\ref{local-same-norm}(1), there exists
$\rho_2(a)>0$ such that $|F(a')|=|F(a)|$, $|Q(a')|=|Q(a)|$,
$|Q'(a')|=|Q'(a)|$ whenever $a'\in B_{\rho(a)}(a)$. Take $r(a)>0$ with
\begin{align*}
&r(a)<\\
&\min \left\{\rho_1(a),\rho_2(a), p^{-s}, p^l|f'(a)|,
\frac{|Q(a)|\cdot|f'(a)|}{|Q'(a)|},
p^{-2s}|Q(a)|\cdot |f'(a)|^2 \right \}.
\end{align*}
 Then the inequality in the claim
holds for
all $a'\in B_{r(a)}(a)$. Finally, 
if $Q'(a)=0$, we apply Lemma~\ref{local-same-norm}(2). The proof is
similar.

Now we  prove the proposition. It is clear that $X =
\bigcup_{a\in X}B_{r(a)}(a)$. Since $X$ is compact, we can find a
finite subcover $X =\bigcup_{k=1}^n B_{r(a_k)}(a_k)$. Let $r_0=
\min_{1\leq k\leq n}r(a_k)$ and $t=\log_p r_0$.
Suppose that  $a_{t,i}\in S_t$, and $f(D_{t,i})\subset
D_{t,j}$. By the definition of $\rho_1$, we know that $s=s(a_{t,i})$
is the least nonnegative integer such that
$P(p^sx+a_{t,i})-(p^sy+a_{t,j})Q(p^sx+a_{t,i})\in
\mathbb{Z}_p[x,y]$. By the definition of $r(a)$, we have
\[p^t<\min\{p^{-s}, p^l|f'(a_{t,i})|, \frac{|Q(a_{t,i})|
\cdot|f'(a_{t,i})|}{|Q'(a_{t,i})|},
p^{-2s}|Q(a_{t,i})|\cdot|f'(a_{t,i})|^2\}.\] This holds for all $a_{t,i}\in S_t$. Therefore, $G^{\ast}(f,p^t)=G(f,p^t)$.
\end{proof}

The following proposition explains the idea behind the 
the definition we chose for subsidiary digraphs.

\begin{prop}\label{bij}
Suppose that  $f'(x)$ has no root in $X$. If $(A_{t,i}, A_{t,j})\in
E(G^{\ast}(f,p^t))$, then
$f:B_{p^t/|f'(a_{t,i})|}(a_{t,i})\rightarrow D_{t,j}$ is a
bijection.
\end{prop}

Before the proof of the proposition we recall the following well known lemma.

\begin{lem}[Hensel's Lemma]\label{hensel}
Let $F(x)$ be a polynomial with coefficients in $\mathbb{Z}_p$. Let
$a\in \mathbb{Z}_p$ such that $|F(a)|<|F'(a)|^2$. Then there exists
a unique root $a'\in \mathbb{Z}_p$ of $F(x)$ such that $|a'-a|\leq
|F(a)|/|F'(a)|$.
\end{lem}

\begin{proof}[Proof of Proposition~\ref{bij}]
Note that $p^t/|f'(a_{t,i})|<p^l$. So $f$ is locally scaling in the
ball $B_{p^t/|f'(a_{t,i})|}(a_{t,i})$ with local scalar
$|f'(a_{t,i})|$. Hence $f(B_{p^t/|f'(a_{t,i})|})\subset D_{t,j}$
and $f$ is injective.

On the other hand, we show that every point in $D_{t,j}$ has a
inverse image in $B_{p^t/|f'(a_{t,i})|}$. Let $b=f(a_{t,i})$ and
let $b'$ be any other point in $D_{t,j}$.
Let $F(x) = P(p^sx+a_{t,i})-b'Q(a_{t,i})$. Then 
\[F(x)=T_{a_{t,i},
a_{t,j}, s}(x, p^{-s}(b'-a_{t,j}))\in \mathbb{Z}_p[x].\] Clearly,
\[|F(0)|=|P(a_{t,i}-b'Q(a_{t,i})|=|b'-f(a_{t,i})|\cdot|Q_{t,i}|<p^s|Q(a_{t,i})|\leq
1.\] Note that
\begin{align*}
|F'(0)|&=|p^s(P'(a_{t,i})-b'Q'(a_{t,i}))|\\&=p^{-s}|Q(a_{t,i})f'(a_{t,i})+(b'-f(a_{t,i})Q'(a_{t,i}))|\end{align*}
and $|(b'-f(a_{t,i})Q'(a_{t,i})|\leq
p^t|Q'(a_{t,i})|<|Q(a_{t,i})f'(a_{t,i})|$. So
$|F'(0)|=p^{-s}|Q(a_{t,i})f'(a_{t,i})|$ and $|F(0)|\leq
p^t|Q(a_{t,i})|<|F'(0)|^2$. By Hensel's lemma, there exists a unique
root $x'\in \mathbb{Z}_p$ of $F(x)$ such that $|x'|\leq
|F(0)|/|F'(0)|\leq p^{t+s}/|f'(a_{t,i})|$. Thus $a' = p^sx'+a_{t,i}$
satisfies $f(a')=b'$ and $|a'-a_{t,i}|=p^{-s}|x'|\leq
p^t/|f'(a_{t,i})|$, as desired.
\end{proof}

\begin{cor}\label{cor-bij}
Let $t\leq t_0$, the intrinsic level. If $(A_{t,i}, A_{t,j})\in E(G(f,p^t))$, then $f:B_{p^t/|f'(a_{t,i})|}(a_{t,i})\rightarrow D_{t,j}$ is a
bijection.
\end{cor}

\begin{proof}
Let $\{a_{t_0, i_0}\}=B_{p^{t_0}}(a_{t,i})\cap S_{t_0}$ and let $D_{t_0, j_0}$ be the ball of radius $p^{t_0}$ containing $D_{t,j}$. Then $(A_{t_0, i_0}, A_{t_0, j_0})\in E(G(f, p^{t_0}))= E(G^{\ast}(f, p^{t_0}))$. By Proposition ~\ref{bij}, $f:B_{p^{t_0}/|f'(a_{t_0, i_0})|}(a_{t_0, i_0})\rightarrow D_{t_0, j_0}$ is a bijection. Recall that $|f'(a)|=|f'(a_{t_0, i_0})|$ whenever $|a-a_{t_0, i_0}|\leq r$. In particular, $|f'(a_{t,i})|=|f'(a_{t_0, i_0})|$. Note that $B_{p^{t_0}/|f'(a_{t_0, i_0})|}(a_{t_0, i_0})$ is the union of $p^{t_0-t}$ disjoint balls of radius $p^t/|f'(a_{t_0, i_0})|$. Write \[B_{p^{t_0}/|f'(a_{t_0, i_0})|}(a_{t_0, i_0})=\bigcup_{s=1}^{p^{t_0-t}}B_s \] where $B_1=B_{p^{t}/|f'(a_{t,i})|}(a_{t,i})$. Similarly, we can write \[D_{t_0, j_0}=\bigcup_{s=1}^{p^{t_0-t}}B'_s\] where every $B'_s$ is a ball of radius $p^t$ and $B'_1=D_{t,j}$. Each $B_s$ maps into one of $\{B'_s\}$. So there is a one-to-one correspondence between the balls $\{B_s\}$ and $\{B'_s\}$. In particular, $f:B_1\rightarrow B'_1$ is a bijection.
\end{proof}

Now we explore  how subsidiary digraphs  help us to study the structure
of rational functions. First, we consider the case when $f$ is a local isometry on $X$. In this case, Corollary ~\ref{cor-bij} gives:

\begin{cor}\label{bij-isom}
Suppose that  $f$ is a local isometry on $X$. Let $t\leq t_0$, the intrinsic level. If $(A_{t,i}, A_{t,j})\in
E(G(f,p^t))$, then $f:D_{t,i}\rightarrow D_{t,j}$ is a
bijection.
\end{cor}

By the definition of $G$, every vertices has outdegree 1. So there
exists at least one cycle in $G$. Given any subgraph $K$ of
a graph $G(f, p^t)$, let $\widetilde{K}$ denote the subset of $\mathbb Q_p$ consisting of the
union of balls corresponding to the vertices of $K$.\\

\begin{prop}~\label{m-p-comp1}
Suppose that  $f$ is a local isometry on $X$. Let $t\leq t_0$ and let $Y=\widetilde{K}$ for some $K$ that is
a union of any collection of cycles in $G(f, p^t)$. Then $f|_Y:
Y\rightarrow Y$ is a well defined invertible local isometry; and,
\textit{a priori}, $f|_Y$ is measure-preserving.
\end{prop}

\begin{proof}
By Corollary~\ref{bij-isom},
$f|_{\mathcal{\widetilde{C}}}:\mathcal{\widetilde{C}}\rightarrow\mathcal{\widetilde{C}}$ is an
invertible local isometry on each cycle $\mathcal{C}$ in
$G(f,p^t)$. Since each vertex in $G$ has outdegree at
most 1, any two cycles in the digraph must be disconnected. Hence,
$f|_Y:Y\rightarrow Y$ is an invertible local isometry. By
Theorem~\ref{m-p}, $f|_Y$ is measure-preserving.
\end{proof}

The proposition above gives us a systematic way to find measure-preserving components of $f$ in $X$ when $f$ is locally isometric. This proposition can be extended to the general case.\\

\begin{prop}~\label{m-p-comp2}
Suppose that  $f'(x)$ has no roots in $X$. Let $t\leq t_0$ and let $Y=\widetilde{K}$ for some $K$ that is
a union of any collection of cycles in $G(f, p^t)$. Then $f|_Y:Y\rightarrow Y$ is measure-preserving if and only if $Y=\widetilde{K'}$ for some $K'$ that is
a union of any collection of cycles in $G(f, p^{t-1})$.
\end{prop}

\begin{proof}
First, if $f|_Y:Y\rightarrow Y$ is measure-preserving, then by Theorem ~\ref{m-p}, $G(f|_Y, p^s)$ must be a disjoint union of cycles for all $s\leq l$. In particular, $G(f|_Y, p^{t-1})$ is a union of cycles.

On the other hand, if $G(f|_Y, p^{t-1})$ is a union of cycles, we show that $f|_Y$ is measure-preserving. By Proposition ~\ref{m-p-comp1}, we only need to show that $f|_Y$ is a local isometry; or equivalently, $|f'(a_{t-1,i})|=1$ for all those $a_{t-1,i}$ in $Y$.

Suppose $|f'(a_{t-1,i_0})|<1$ for some $i_0$. Let $D_{t-1,j_0}$ denote the ball $B_{p^{t-1}}(f(a_{t-1,i_0}))$. By Corollary ~\ref{cor-bij}, $f:B_{p^{t-1}/|f'(a_{t-1,i_0})|}(a_{t-1,i_0})\rightarrow D_{t-1, j_0}$ is a bijection. Note that $B_{p^t}(a_{t-1, i_0})\subseteq B_{p^{t-1}/|f'(a_{t-1, i_0})|}(a_{t-1, i_0})\cap Y$. Let $S_{t-1}\cap B_{p^t}(a_{t-1, i_0})=\{a_{t-1, i_0}, a_{t-1, i_1},..., a_{t-1, i_{p-1}}\}$. Then $(A_{t-1, i_s}, A_{t-1, j_0})\in E(G^{\ast}(f, p^{t-1}))$, $s=0,1,...,p-1$. Hence $d^{-}(A_{t-1, j_0})\geq p>1$. So $Y$ cannot be a disjoint union of cycles in $G(f, p^{t-1})$, a contradiction. Therefore, $f|_Y$ is a local isometry, as desired.
\end{proof}

\begin{exmp}
Let $X$ and $f$ be the same as in example ~\ref{examp}. It is not hard to see that $|f'(x)|=|(x^2+1)/x^2|=1$ for all $x\in X$. By Proposition ~\ref{criteria}, $f$ is a local isometry. We can take $r=7^{-1}$ and $l=-1$. Take $t=-2$ and select representatives
\[S_{-2}=\{2,9,16,23,30,37,44,5,12,19,26,33,40,47\}.\] Then the subsidiary digraph $G^{\ast}(f, 7^{-2})$ coincides with the digraph $G(f, 7^{-2})$. By Proposition ~\ref{m-p-comp1}, the union of any collection of cycles in $G(f, 7^{-2})$ corresponds to a measure-preserving component. For example, take 
\[Y=B_{7^{-2}}(2)\cup B_{7^{-2}}(9)\cup B_{7^{-2}}(23)\cup B_{7^{-2}}(26)\cup B_{7^{-2}}(40)\cup B_{7^{-2}}(47).\] Then $Y$ corresponds to a cycle of length 6. So $f|_Y:Y\rightarrow Y$ is measure-preserving. It is also easy to see that $f$ is measure-preserving on $X$ since $G(f, 7^{-2})$  is itself a union of cycles.

\end{exmp}

\begin{exmp}
Let \[X = \mathbb{Z}_3-(B_{3^{-2}}(4)\cup B_{3^{-2}}(5))\text{ and }f(x) = \frac{2x^3+x^2+x}{x^2+1}.\] Then $f'(x) = \frac{2x^4+5x^2+2x+1}{(x^2+1)^2}$. It is not hard to check that $|f'(x)|\leq 1$ on $X$ and $f'$ has no zero on $X$. So $f$ is 1-Lipschitz. We can take $r=3^{-2}$ and  $l=-2$. Take $S_{-2}=\{0,1,2,3,6,7,8\}$ and \[S_{-3}=\{0,1,2,3,6,7,8,9,10,11,12,15,16,17,18,19,20,21,24,25,26\}.\] The corresponding digraphs and subsidiary digraphs at $t=-2, -3$ are as shown in Figure~\ref{F:three}.
\begin{figure}[h]

\begin{center}
    \psfrag{A}[]{$Y_1$}
    \psfrag{B}[]{$Y_2$}
    \psfrag{C}[]{$G^{\ast}(f,3^{-2})=G(f,3^{-2})$}
    \psfrag{D}[]{$G(f,3^{-3})$}
 \includegraphics[width=12cm]{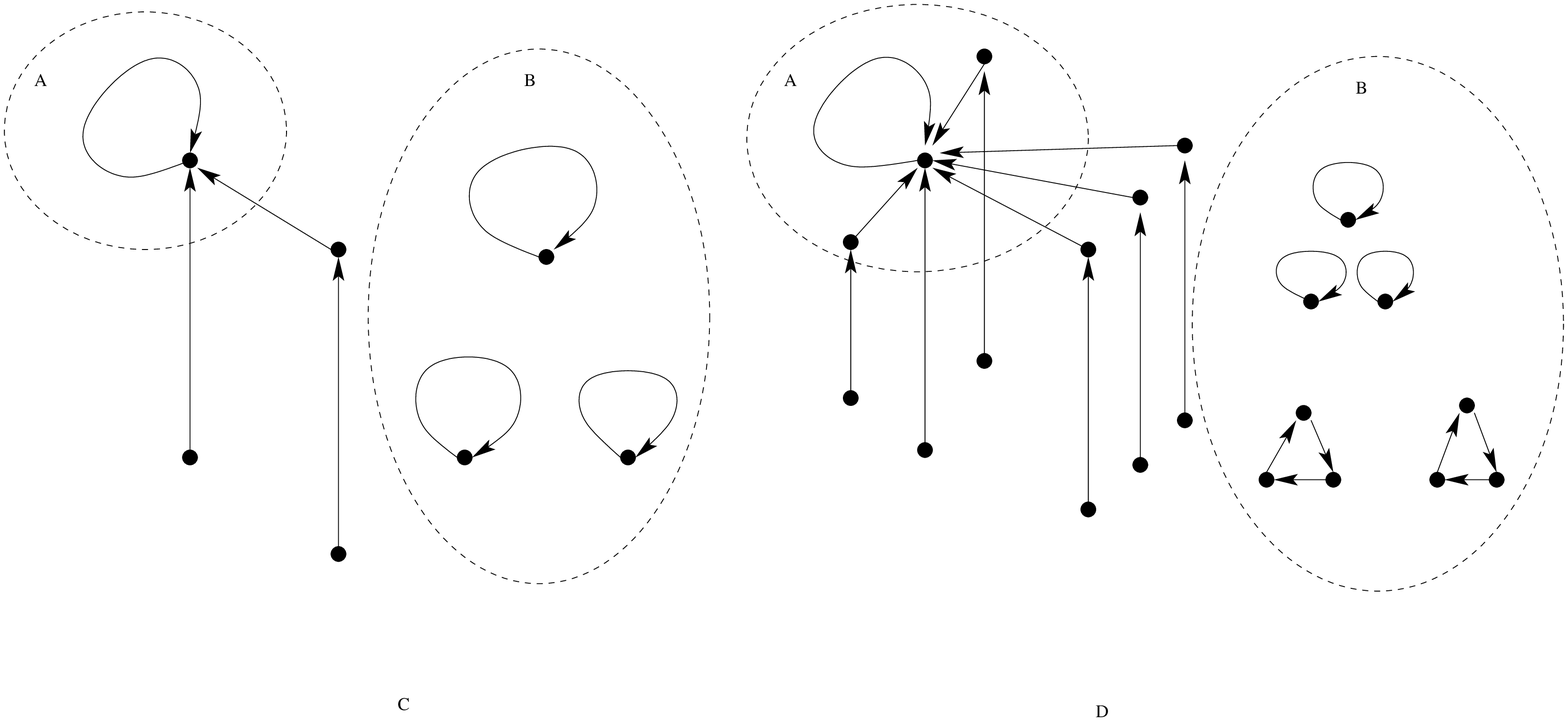}\\

\end{center}
\caption { }
 \label{F:three}
\end{figure}
We see that $G^{\ast}=G$ at $t=-2$. So $t_0=-2$. Let $Y_1$ and $Y_2$ be the subset of $X$ associated to the corresponding subgraphs as above; i.e., $Y_1=B_{3^{-2}}(8)$ and $Y_2=B_{3^{-1}}(0)$. In $G(f,3^{-2})$, both subgraphs consist of union of cycles. However, in $G(f,3^{-3})$, $Y_1$ does not correspond to union of cycles. By Proposition ~\ref{m-p-comp2}, $f|_{Y_2}$ is measure-preserving but $f|_{Y_1}$ is not.

\end{exmp}

\noindent\textbf{Remark.}  Propositions ~\ref{m-p-comp1} and ~\ref{m-p-comp2} provide a systematic way to find subset of $X$ on which $f$ is measure-preserving. In addition, the algorithm can be implemented by a computer. All computations are direct given initial data ($X,f,r,A_{t,i}, S_t$). So we only need to show that $r$ can be determined in a reasonable time. This is discussed in the Appendix.

%%%%

\section{A Characterization of Invertible Locally Isometric Rational Functions over $\mathbb{Q}_p$}\label{S:qp}

Let $f$ be a rational function over $\mathbb{Q}_p$. Write $f(x)=p^{\alpha}\frac{P_1(x)}{Q_1(x)}$ where $P_1(x) = a_mx^m+a_{m-1}x^{m-1}+\cdots+a_0$ and $Q_1(x)=b_nx^n+b_{n-1}x^{n-1}+\cdots+b_0$ with $a_m, b_n\in \mathbb{Z}_p^{\times}$. Assume $Q_1$ has no root on $\mathbb{Q}_p$. Write $f'(x) = p^{\alpha'}\frac{P_2(x)}{Q_2(x)}$ where $P_2(x)=a'_{m'}x^{m'}+\cdots +a'_0$ and $Q_2(x)=b'_{n'}x^{n'}+\cdots +b'_0$ with $a'_{m'}, b'_{n'}\in \mathbb{Z}_p^{\times}$.

\begin{prop}\label{inv-iso-qp}
If $f(x)$ is an invertible local isometry on $\mathbb{Q}_p$, we must have $\alpha=0$ and $m=n+1$.
\end{prop}

\begin{prop}\label{m-p-qp}
If $f(x)$ is measure-preserving on $\mathbb{Q}_p$, we must have $\alpha=0$ and $m=n+1$.
\end{prop}

To prove these results, we need the following two lemmas.

\begin{lem}\label{N}
For any polynomial $F(x)=c_sx^s+\cdots +c_0\in \mathbb{Q}_p[x]$ with $c_s\in \mathbb{Z}_p$, there exists a positive integer $N$ such that $|F(x)|=|x|^s$ whenever $|x|\geq p^N$.
\end{lem}

\begin{proof}
Pick $N\in \mathbb{Z}_{>0}$ such that $p^N>\max_{0\leq i\leq s-1}|c_i|$. Then $|c_{s-1}x^{s-1}+\cdots+c_0|\leq \max_{0\leq i\leq s-1}|c_ix^i|<|x|^s=|c_sx^s|$ when $|x|\leq p^N$. Hence $|F(x)|=|c_sx^s|=|x|^s$ whenever $|x|\geq p^N$.
\end{proof}

\begin{lem}\label{l}
If $F(x)\in \mathbb{Q}_p[x]$ and $F(x)\neq 0$ for all $x\in \mathbb{Q}_p$, then there exists a positive integer $l_0$ such that $|F(x)|\geq p^{l_0}$ for all $x\in \mathbb{Q}_p$.
\end{lem}

\begin{proof}
Without loss of generality, we can assume $F(x)\in \mathbb{Z}_p[x]$. Note that $\mathbb{Z}_p$ is the inverse limit of $\mathbb{Z}/p^l\mathbb{Z}$ via the natural projection maps $\tau: \mathbb{Z}/p^l\mathbb{Z}\rightarrow \mathbb{Z}/p^{l-1}\mathbb{Z}$. For any $l\in \mathbb{Z}_{>0}$, we can view $F(x)$ as a polynomial over $\mathbb{Z}/p^l\mathbb{Z}$. Define $H_l=\{x\in \mathbb{Z}/p^l\mathbb{Z}\,|\, F(x)=0\}$. Clearly, $\tau(H_l)\subset H_{l-1}$. Suppose for every $l>0$, there exists $x\in \mathbb{Z}_p$ such that $|F(x)|\geq p^{-l}$. Then $H_l\neq 0$ for all $l>0$. So the inverse limit of $H_l$'s is nonzero. But every element in $\underleftarrow{H_l}$ is a zero of $F(x)$, a contradiction.
\end{proof}

\begin{proof}[Proof of Proposition~\ref{inv-iso-qp}]
Since $f$ is locally isometric, we must have $|f'(x)|=1$ for all $x\in \mathbb{Q}_p$. 
By Lemma ~\ref{N}, there exists $N_0\in \mathbb{Z}_{>0}$ such that $|P_2(x)|=|x|^{m'}$ and $|Q_2(x)|=|x|^{n'}$ whenever $|x|\geq p^{N_0}$. Hence, if $|x|\geq p^{N_0}$,
\[|f'(x)|=|p^{\alpha'}|\frac{|x|^{m'}}{|x|^{n'}}=p^{-\alpha'}|x|^{m'-n'}.\] So we must have $\alpha'=0$ and $m'=n'$.\\
Note that \[f'(x)=p^{\alpha}\frac{P_1'(x)Q_1(x)-P_1(x)Q'_1(x)}{Q_1^2(x)}.\] If $m\leq n$, then
\[\deg \parb{P_1'(x)Q_1(x)-P_1(x)Q'_1(x)}\leq m+n-1<2n=\deg Q_1^2(x).\] If $m\geq n+2$, then the highest term in $P'_1(x)Q_1(x)-P_1(x)Q'_1(x)$ is $(m-n)a_mb_nx^{m+n-1}$. So \[\deg \parb{P_1'(x)Q_1(x)-P_1(x)Q'_1(x)}= m+n-1>2n=\deg Q_1^2(x).\] So the only possible case is $m=n+1$. In this case, the highest term in the nominator is $a_mb_nx^{2n}$. So we must have $\alpha=0$ because $\alpha'=0$.
\end{proof}

\begin{proof}[Proof of Proposition ~\ref{m-p-qp}]
By Lemma ~\ref{l}, there exists $l_0\in \mathbb{Z}_{>0}$ such that $|Q_1(x)|\geq p^{l_0}$ for all $x\in \mathbb{Q}_p$. By Lemma ~\ref{N}, there exists $N_0\in \mathbb{Z}_{>0}$ such that $|P_1(x)|=|x|^m$ and $|Q_1(x)|=|x|^n$ whenever $|x|\geq p^{N_0}$. If $m\neq n$ or $\alpha\neq 0$, then it must belong to one of the following cases.

\emph{Case 1}: $\{m\leq n\}$ or $\{m=n+1$ and $\alpha>0\}$ \\
Let
\[N=
\left\{ \begin{array}{ll}N_0 & \textrm{if } m=n+1 \textrm{ and } \alpha>0\\
\max\{N_0, \frac{-\alpha-1}{n+1-m}\}& \textrm{if }m\leq n. \end{array} \right.\]
For any $x$ with absolute value $\geq p^N$, we have
\[|f(x)|=|p^\alpha|\frac{|x|^m}{|x|^n}=p^{-\alpha}|x|^{m-n}< |x|.\]
For any $x$ with absolute value $< p^N$, we have
\[|P_1(x)|\leq \max_{0\leq i\leq m}|a_i||x|^i\leq \max_{0\leq i\leq m}|a_i|p^{N_i}.\]
So \[|f(x)|\leq |p^{\alpha}|\frac{\max_{0\leq i\leq m}|a_i|p^{Ni}}{p^{l_0}}=p^{-\alpha+l_0}\max_{0\leq i\leq m}|a_i|p^{Ni}=: p^{l_1} .\]
Let $N_1=\max\{N, l_1\}$. Then $B_{p^{N_1+1}}(0)\subset f^{-1}(B_{p^{N_1}}(0))$. But $\mu(B_{p^{N_1+1}}(0))=p^{N_1+1}>p^{N_1}=\mu(B(p^{N_1})(0))$. So $f$ is not measure-preserving, a contradiction.

\emph{Case 2}: $\{m-n\geq 2$ or $m=n+1, \alpha<0\}$\\
Let
\[N=
\left\{ \begin{array}{ll}N_0 & \textrm{if } m=n+1 \textrm{ and } \alpha<0\\
\max\{N_0, \frac{-\alpha+1}{n+1-m}\}& \textrm{if }m-n\geq 2. \end{array} \right.\]
Then, by a similar argument as in Case 1, we have $|f(x)|=p^{-\alpha}|x|^{m-n}>|x|$ when $|x|\geq p^N$, and
\[
|f(x)|\leq p^{-\alpha+l_0}\max_{0\leq i\leq m}|a_i|p^{Ni}=: p^{l_1}\] when $|x|<p^N$. Let $N_1=\max\{N, l_1\}$. Then $f^{-1}(B_{p^{N_1+1}}(0))\subset B_{p^{N_1}}(0)$. But $\mu(B_{p_{N_1}}(0))=p^{N_1}<p^{N_1+1}=\mu(B_{p^{N_1+1}}(0))$. So $f$ is not measure-preserving, a contradiction.

Therefore, we must have $m=n+1$ and $\alpha =0$.

\end{proof}

From now on we assume $\alpha=0$ and $m=n+1$. By lemma ~\ref{N}, there exists $N\in \mathbb{Z}_{>0}$ such that $p^N>\max_{0\leq i\leq n}|a_n|$, $p^N>\max_{0\leq j\leq n-1}|b_j|$, and $|P_2(x)|=|Q_2(x)|=|x|^{n'}$ whenever $|x|\geq p^{N}$.

\begin{prop}\label{inv-iso-qp2}
 Define $N$ as above. Then $f$ is an invertible local isometry on $\mathbb{Q}_p$ if and only if $f(B_{p^{N-1}}(0))\subset B_{p^{N-1}}(0)$ and $f|_{B_{p^{N-1}}(0)}$ is invertible and locally isometric.
\end{prop}

\begin{proof}
If $|x|\geq p^N$, then $|f(x)|=\frac{|x|^{n+1}}{|x|^n}=|x|$. So $f\parb{S_{p^{N_1}}(0)}\subset S_{p^{N_1}}(0)$ whenever $N_1\geq N$. Now we show that $f|_{S_{p^{N_1}}(0)}$ is always an invertible local isometry. First, by the discussion in the previous proof, we have $|f'(x)|=1$ for all $x\in S_{p^{N_1}}(0)$. So $f|_{S_{p^{N_1}}(0)}$ is a local isometry.\\
To verify the invertibility, we show the following: for any $a\in \mathbb{Z}_p^{\times}$, there exists $x_0\in p\mathbb{Z}_p^{\times}$ such that $f(p^{-N_1}(\frac{b_n}{a_{n+1}}a+x_0))=p^{-N_1}a$. Let $F(x)=p^{(n+1)N}P_1(\frac{b_n}{a_{n+1}}a+x)-ap^{nN}Q_1(\frac{b_n}{a_{n+1}}a+x)\in \mathbb{Z}_p[x]$. We have
\[
|F(0)|= p^{-N}|\sum_{i=0}^{n}p^ia_{n-i}(\frac{b_n}{a_{n+1}}a)^{n-i}-ap^N\sum_{i=0}^{n-1}p^ib_{n-i-1}(\frac{b_n}{a_{n+1}}a)^{n-i-1}|\\
\leq p^{-1},
\]
and
\begin{align*}
|F'(0)|=& |\frac{b_n^n}{a_{n+1}^{n-1}}a^n+p^N\sum_{i=0}^{n-1}(n-i)p^ia_{n-i}(\frac{b_n}{a_{n+1}}a)^{n-i-1}\\
& +ap^N\sum_{i=0}^{n-2}(n-i-1)p^ib_{n-i-1}(\frac{b_n}{a_{n+1}}a)^{n-i-2}|\\
=& 1.
\end{align*}
So, by Lemma ~\ref{hensel}, $F(x)$ has a root $x=x_0$ in $p\mathbb{Z}_p$. Then \[f(p^{-N_1}(\frac{b_n}{a_{n+1}}a+x_0))=p^{-(n+1)N_1}Q(\frac{b_n}{a_{n+1}}a+x_0)F(x_0)+p^{-N_1}a=p^{-N_1}a,\]
as desired. Hence $f|_{S_{p^{N_1}}(0)}$ is always an invertible local isometry.\\
If $f$ is invertible, then $f(B_{p^{N-1}}(0))\cap S_{p^{N_1}}(0)=\emptyset$ for all $N_1\geq N$. So $f(B_{p^{N-1}}(0))\subset B_{p^{N-1}}(0)$. Clearly, $f|_{B_{p^{N-1}}(0)}$ must be invertible and locally isometric.\\
The other direction follows immediately from \[\mathbb{Q}_p= B_{p^{N-1}}(0)\cup \parb{\bigcup_{N_1=N}^{\infty}S_{p^{N_1}}(0)}.\]
\end{proof}

\begin{prop}\label{m-p-qp2}
Define $N$ as above. Then $f$ is measure-preserving on $\mathbb{Q}_p$ if and only if $f(B_{p^{N-1}}(0))\subset B_{p^{N-1}}(0)$ and $f|_{B_{p^{N-1}}(0)}$ is measure-preserving.
\end{prop}

\begin{proof}
In the proof of Proposition ~\ref{inv-iso-qp2}, we have seen that $f(S_{p^{N_1}}(0))\subset S_{p^{N_1}}(0)$ and $f|_{S_{p^{N_1}}(0)}$ is an invertible local isometry whenever $N_1\geq N$. By Theorem ~\ref{m-p}, $f|_{S_{p^{N_1}}(0)}$ is measure-preserving. \\
Suppose $f$ is invertible. If $a\in f(B_{p^{N-1}}(0))\cap S_{p^{N_1}}(0)$ for some $N_1\geq N$, then $f^{-1}(S_{p^{N_1}}(0))\subset S_{p^{N_1}}(0)\cup B_{r_1}(a)$ for some $r_1>0$. Then \[p^{N_1}=\mu(f^{-1}(S_{p^{N_1}}(0)))\geq \mu(S_{p^{N_1}}(0))+\mu(B_{r_1}(a))=p^{N_1}+r_1>p^{N_1},\] a contradiction. Hence, $f(B_{p^{N-1}}(0))\cap S_{p^{N_1}}(0)=\emptyset$ for all $N_1\geq N$. So $f(B_{p^{N-1}}(0))\subseteq B_{p^{N-1}}(0)$. Clearly, $f|_{B_{p^{N-1}}(0)}$ must be measure-preserving.\\
The other direction follows immediately from \[\mathbb{Q}_p= B_{p^{N-1}}(0)\cup \parb{\bigcup_{N_1=N}^{\infty}S_{p^{N_1}}(0)}.\]
\end{proof}

Once $f$ is restricted on a compact open subset of $\mathbb{Q}_p$, we can apply Theorem ~\ref{m-p}. The following corollary shows that measure-preserving and invertible locally isometric are the same thing over $\mathbb{Q}_p$.

\begin{cor}\label{m-p-inv-iso}
Let $f$ be a locally 1-Lipschitz rational function on $\mathbb{Q}_p$. Then $f$ is measure-preserving if and only if it is invertible and locally isometric.
\end{cor}

\begin{proof}
By Proposition ~\ref{inv-iso-qp2} and \ref{m-p-qp2}, we only need to show that $f|_{B_{p^{N-1}}(0)}$ is measure-preserving if and only if it is invertible and locally isometric. This follows immediately from Theorem ~\ref{m-p}.
\end{proof}

Now consider the special case when $P_1(x), Q_1(x)\in \mathbb{Z}_p[x]$. In this case, we can always take $N=1$. Then we obtain the following corollary.

\begin{cor}\label{cor-inv-iso-qp}
Suppose that  $f(x)=\frac{P_1(x)}{Q_1(x)}$ where 
\begin{align*}P_1(x)&=a_{n+1}x^{n+1}+a_nx^n+\cdots+a_0\in \mathbb{Z}_p[x],\\
 Q_1(x)&=b_nx^n+\cdots+b_0\in \mathbb{Z}_p[x]\end{align*} and $a_{n+1}, b_n\in \mathbb{Z}_p^{\times}$. Assume $Q_1(x)$ has no root on $\mathbb{Q}_p$. Then $f$ is an invertible local isometry on $\mathbb{Q}_p$ if and only if  $f(\mathbb{Z}_p)\subset \mathbb{Z}_p$ and $f|_{\mathbb{Z}_p}$ is invertible and locally isometric.
\end{cor}

\begin{exmp}\label{exmp-inv-qp}
Consider the rational function $f(x)=\frac{x^4+x^3+2x^2+1}{x^3-x+1}$ over $\mathbb{Q}_3$. By Corollary ~\ref{cor-inv-iso-qp}, to show that $f$ is an invertible local isometry, we only need to work on $\mathbb{Z}_3$. Clearly, $f(\mathbb{Z}_3)\subseteq \mathbb{Z}_3$. $f'(x)=\frac{x^6-5x^4+2x^3-2x^2+4x+1}{(x^3-x+1)^2}$. It is not hard to see that $|f'(x)|=1$ for all $x\in \mathbb{Z}_3$. So $f|_{\mathbb{Z}_3}$ is a local isometry on $\mathbb{Z}_3$. We can take $r=3^{-1}$ and $l=-1$. Then both $G(f|_{\mathbb{Z}_3}, 3^{-1})$ and $G^{\ast}(f|_{\mathbb{Z}_3}, 3^{-1})$ consist of a single cycle of length 3. So $t_0=-1$. By Proposition ~\ref{m-p-comp1},  $f|_{\mathbb{Z}_3}$ is an invertible local isometry. Hence, $f$ is an invertible local isometry on $\mathbb{Q}_3$. By Corollary ~\ref{m-p-inv-iso}, we know that $f$ is also measure-preserving.
\end{exmp}

%\centerline{\bf Appendix}
\appendix
\section{ }

\subsection{Proofs of Proposition 2.1 and 2.2}

Write $f(x)=p^{\alpha}\frac{P_1(x)}{Q_1(x)}$ where $P_1(x) = a_mx^m+a_{m-1}x^{m-1}+\cdots+a_0$ and $Q_1(x)=b_nx^n+b_{n-1}x^{n-1}+\cdots+b_0$ with $a_m, b_n\in \mathbb{Z}_p^{\times}$.

\begin{proof}[Proof of Proposition~\ref{no-minimal}]

By Lemma ~\ref{l}, there exists $l_0\in \mathbb{Z}_{>0}$ such that $|Q_1(x)|\geq p^{l_0}$ for all $x\in \mathbb{Q}_p$. By Lemma ~\ref{N}, there exists $N_0\in \mathbb{Z}_{>0}$ such that $|P_1(x)|=|x|^m$ and $|Q_1(x)|=|x|^n$ whenever $|x|\geq p^{N_0}$.

\emph{Case 1}: $\{m\leq n\}$ or $\{m=n+1$ and $\alpha>0\}$ \\
Let
\[N=
\left\{ \begin{array}{ll}N_0 & \textrm{if } m=n+1 \textrm{ and } \alpha>0\\
\max\{N_0, \frac{-\alpha}{n+1-m}\}& \textrm{if }m\leq n. \end{array} \right.\]
For any $x$ with absolute value $\geq p^N$, we have
\[|f(x)|=|p^\alpha|\frac{|x|^m}{|x|^n}=p^{-\alpha}|x|^{m-n}\leq |x|.\]
For any $x$ with absolute value $< p^N$, we have
\[|P_1(x)|\leq \max_{0\leq i\leq m}|a_i||x|^i\leq \max_{0\leq i\leq m}|a_i|p^{N_i}.\]
So \[|f(x)|\leq |p^{\alpha}|\frac{\max_{0\leq i\leq m}|a_i|p^{Ni}}{p^{l_0}}=p^{-\alpha+l_0}\max_{0\leq i\leq m}|a_i|p^{Ni}=: p^{l_1} .\]
Now it is clear that $f(B_{p^{l_1}}(0))\subseteq B_{p^{l_1}}(0)$. So $f$ is not minimal.\\

\emph{Case 2}: $\{m>n$ and $\alpha\leq 0\}$\\
If $|x|\geq p^{N_0}$, then \[ |f(x)|=|p^{\alpha}|\frac{|x|^m}{|x|^n}=p^{-\alpha}|x|^{m-n}\geq p^{N_0}. \]
So $\{f^{(n)}(x)\,|\, n\in \mathbb{Z}_{>0}\}\cap B_{p^{N-1}}(0)=\emptyset$. Hence $f$ is not minimal.\\

\emph{Case 3}: $\{m-n\geq 2$ and $\alpha>0\}$\\
Let $N = \max\{N_0, \frac{\alpha}{m-n-1}\}$. If $|x|\geq p^N$, then $|f(x)|=p^{-\alpha}|x|^{m-n}\geq p^N$. So $\{f^{(n)}(x)\,|\, n\in \mathbb{Z}_{>0}\}\cap B_{p^{N-1}}(0)=\emptyset$. Therefore, $f$ is not minimal.

\end{proof}

\begin{proof}[Proof of Proposition~\ref{no-ergodic}]
By Proposition ~\ref{m-p-qp}, we must have $\alpha=0$ and $m=n+1$. Define $N$ as in Proposition ~\ref{m-p-qp2}. We already seen in the proof of Proposition ~\ref{m-p-qp2} that $S_{p^{N}}(0)$ is $f$-invariant. So $f$ is not ergodic.
\end{proof}

\subsection{An Algorithm to Determine $r$}
In this section we present an algorithm to compute the radius $r$ involved in Definition ~\ref{unif-local-scal}.\\

First, consider any polynomial $F(x)\in \mathbb{Z}_p[x]$. Suppose $F(x)\neq 0$ on a compact open subset $X\subset \mathbb{Q}_p$. Then $|F(x)|$ is bounded below on $X$. Below is an algorithm to compute a lower bound for $|F(x)|$. We denote this lower bound by $b(F)$.

\begin{enumerate}
\item Initial Data: $X=\bigsqcup_{i=1}^mD_{s,i}$ where $D_{s,i}$ are balls of radius $p^s$. We can assume that $s<0$.
\item Select representatives $S_t$ for $t\leq s$. (See Section ~\ref{S:subsidiary} for the definition of \emph{representatives}.)
\item Let $t=s$.
\item Compute $F(a_{t,i})$ for $a_{t,i}\in S_t$.
\item If $F(a_{t,i})\equiv 0 (\textrm{mod }p^{-t})$ for some $a_{t,i}\in S_t$, then replace $t$ by $t-1$ and repeat step (4). If $F(a_{t,i})\not\equiv 0 (\textrm{mod }p^{-t})$ for all $a_{t,i}\in S_t$, then set $b(F)=p^{t+1}$.
\end{enumerate}
Since $|F(x)|$ is bounded below, this process will terminate eventually.\\

Now let $f(x) = \frac{P(x)}{Q(x)}$ for $P(x), Q(x)\in \mathbb{Z}_p[x]$ and assume that $Q(x)$ has not root on $X$. We also assume that $f'(x)$ has no zero on $X$. Write \[f(x)-f(y)=\frac{x-y}{Q(x)Q(y)}T(x,y).\]
and write $T_1(x)=T(x,y)$. Note that $T_1(x)=Q(x)^2f'(x)$. So $T_1(x)$ has no root on $X$. Let $r=\min\{p^{-1}b(Q), p^{-1}b(T_1)\}$. The following lemma tells us this $r$ is the one satisfies the conditions in Definition ~\ref{unif-local-scal}.

\begin{lem}
For any $a\in X$ and $x,y\in B_r(a)$, we have \[|f(x)-f(y)|=|f'(a)|\cdot|x-y|.\]
\end{lem}

\begin{proof}

Write $T(x,y)=\sum a_{i,j}x^iy^j\in \mathbb{Z}_p[x,y]$. If $x,y\in B_r(a)$, then
\begin{align*}
|T(x,y)-T(a,a)|=& |\sum a_{i,j}(x^i-a^i)y^j+\sum a_{i,j}a^i(y^j-a^j)|\\
\leq&\max\{ |(x-a)|\cdot|\sum a_{i,j} (x^{i-1}+x^{i-2}a+\cdots+a^{i-1})y^j|,\\
& |y-a|\cdot|\sum a_{i,j}a^i(y^{j-1}+y^{j-2}a+\cdots+a^{j-1})|\}\\
\leq& r<b(T_1)\leq |T_1(a)|=|T(a,a)|
\end{align*}
Hence $|T(x,y)|=|T(a,a)|$ whenever $x,y\in B_r(a)$. Similarly, $|Q(x)|=|Q(y)|=|Q(a)|$ whenever $x,y\in B_r(a)$. Therefore, \[|f(x)-f(y)|=\frac{|x-y|}{|Q(x)|\cdot|Q(y)|}|T(x,y)|=\frac{|x-y|}{|Q(a)|^2}|T(a,a)|=|x-y|\cdot|f'(a)|,\] as desired.

\end{proof}

%%%%%%%%%%%%%%%%%%%%%%%%%%%%%%%%%%%%%%%%%%%%%

%Bibliography
\bibliographystyle{amsalpha}
\bibliography{Digraph_Bib}

%%%%%%%%%%%%%%%%%%%%%%%%%%%%%%%%%%%%%%%%%%%%%

\end{document}